\documentclass[12pt,a4paper]{article}

\usepackage[utf8]{inputenc}
\usepackage[ngerman,english]{babel} 
\usepackage{eurosym}
\usepackage{marginnote}
\usepackage{amsmath}
\usepackage{amsthm}
\usepackage{amsfonts}
\usepackage{makeidx}
\usepackage{amssymb}
\usepackage{graphicx}
\usepackage{tabularx}
\usepackage{epsfig}
\usepackage{color}
\usepackage{enumerate}
\usepackage{setspace}
\usepackage{geometry}

\newtheorem{definition}{Definition}[section]

\newtheorem{lemma}[definition]{Lemma}

\newtheorem{proposition}[definition]{Proposition}
\newtheorem{algorithm}[definition]{Algorithm}
\theoremstyle{remark}
\newtheorem{example}[definition]{Example}

\newcommand{\EE}{{\mathbb E}}
\newcommand{\N}{{\mathbb N}}

\newcommand{\R}{{\mathbb R}}

\newcommand{\V}{{\mathbb V}}

\begin{document}

\title{Fast orthogonal transforms for multi-level quasi-Monte Carlo integration}

\author{Christian Irrgeher\thanks{Christian Irrgeher, Institute of Financial Mathematics, Johannes Kepler University Linz, Altenbergerstra{\ss}e 69, A-4040 Linz, Austria. e-mail: {\tt christian.irrgeher@jku.at} \quad The author is supported by the Austrian Science Foundation (FWF), Project P21943 .}~~and Gunther Leobacher\thanks{Gunther Leobacher, Institute of Financial Mathematics, Johannes Kepler University Linz, Altenbergerstra{\ss}e 69, A-4040 Linz, Austria. e-mail: {\tt gunther.leobacher@jku.at} \quad The author is partially supported by the Austrian Science Foundation (FWF), Project P21196.}}

\maketitle

\begin{abstract}
We combine a generic method for finding fast orthogonal transforms for a given quasi-Monte Carlo integration problem with the multilevel Monte Carlo method. It is shown by example that this combined method can vastly improve the efficiency of quasi-Monte Carlo.
\end{abstract}


\section{Introduction}\label{sec:introduction}

Many simulation problems from finance and other applied fields can be written in the form $\EE(f(X))$, where $f$ is a  measurable function on $\R^n$ and $X$ is a standard normal vector, that is, $X=(X_1,\ldots,X_n)$ is jointly normally  distributed with $\EE(X_j)=0$ and $\EE(X_jX_k)=\delta_{jk}$. It is a trivial observation that 
\begin{align}\label{eq:fundamental}
\EE(f(X))=\EE(f(U\!X)) 
\end{align}
for every orthogonal transform $U:\R^n\longrightarrow\R^n$. It has been observed in a number of articles (\cite{abg98,Moskowitz96,leo}) that, while this reformulation does not change the simulation problem from the probabilistic point of view, it does make a -- sometimes big --  difference when quasi-Monte Carlo (QMC) simulation is applied to generate the realizations of $X$.

Prominent examples are supplied by the well-known Brownian bridge \cite{Moskowitz96} and principal component analysis (PCA) \cite{abg98} constructions of Brownian paths which will be detailed in the following paragraphs. Assume we want to calculate an approximation to $\EE(g(B))$ where $B$ is a Brownian motion with index set $[0,T]$. In most applications this can be reasonably approximated by $\EE(\tilde g(B_\frac{T}{n},\ldots,B_\frac{T n}{n}))$, where $\tilde g$ is a corresponding function taking as its argument a {\em discrete Brownian path}, by which we mean a normal vector with covariance matrix
\begin{align*}
\Sigma:=\Big(\frac{T}{n}\min(j,k)\Big)_{j,k=1}^n=\frac{T}{n}\left(\begin{array}{cccccccc}
1&1&1&\ldots&1\\
1&2&2&\ldots&2\\
1&2&3&\ldots&3\\
\vdots&\vdots&\vdots&\ddots&\vdots\\
1&2&3&\ldots&n
\end{array}\right)\,.
\end{align*}

There are three classical methods for sampling from $(B_\frac{T}{n},\ldots,B_\frac{nT}{n})$ given a standard normal vector $X$, namely the forward method, the Brownian bridge construction  and the principal component analysis construction. All of these constructions may be written in the form $(B_\frac{T}{n},\ldots,B_\frac{nT}{n})=A X$, where $A$ is an $n\times n$ real matrix with $ A A^\top=\Sigma$.

For example, the matrix corresponding to the forward method is 
\begin{align}\label{eq:summation-matrix} A=S:=\sqrt{\frac{T}{n}}\left(\begin{array}{cccccccc}
1&0&\ldots&0\\
1&1&\ldots&0\\
\vdots&\vdots&\ddots&\vdots\\
1&1&\ldots&1
\end{array}\right)\,,
\end{align}
while PCA corresponds to  $A=V\!D$, where $\Sigma=V\!D^2V^\top$ is the singular value decomposition of $\Sigma$. A corresponding decomposition for the Brownian bridge algorithm is given, for example, in Larcher, Leobacher \& Scheicher \cite{lalesch}.

It has been observed by Papageorgiou \cite{papa} that $A A^\top = \Sigma$ if and only if $A=SU$ for some orthogonal matrix $U$, so that every linear construction of $(B_\frac{T}{n},\ldots,B_\frac{nT}{n})$ corresponds to an orthogonal transform of $\R^n$. In that sense the forward method corresponds to the identity, PCA corresponds to $S^{-1}V\!D$ and Brownian bridge corresponds to the inverse Haar transform, see Leobacher \cite{leo11}.

Thus our original simulation problem can be written, as
\begin{align*}
\EE(\tilde g(B_\frac{T}{n},\ldots,B_\frac{T n}{n}))=\EE(\tilde g(S X))=\EE(f(X))
\end{align*}
with $f=\tilde g\circ S$. In the context of discrete Brownian paths this corresponds to the forward method. Consequently, the same problem using the Brownian bridge takes on the form $\EE(f(H^{-1} X))$, where $H$ is the matrix of the Haar transform, and has the form $\EE(f(S^{-1}VD X))$, with $S,$ $V,$ $D$ as above, when PCA is used.

Papageorgiou \cite{papa} noted that whether or not the Brownian bridge and PCA constructions enhance the performance of QMC methods depends critically on the integrand $f$ and he provides an example of a financial option where those two methods give much worse results than the forward method. That lead to the idea of searching for orthogonal transform tailored to the integrand. Imai \& Tan \cite{imaitan07} propose a general technique for this problem which they call linear transform (LT) method.
\\

The exact reason why orthogonal transforms might have the effect to make a problem more suitable for QMC is still unknown. Caflish et. al. \cite{cmo97} propose that those transforms diminish the so-called {\em effective dimension} of the problem. Owen \cite{owen2012} provided the concept of {\em effective dimension of a function space}. The least that can be said with confidence is that introducing an orthogonal transform does not introduce a bias and that there are choices (like the identity) that make the problem at least equally well suited for QMC as the original one. 

While applying a suitable orthogonal transform to an integration problem may increase the performance of QMC simulation, there is also a disadvantage: the computation of the orthogonal transform incurs a cost, which in general is of the order $O(n^2)$. For large $n$ this cost is likely to swallow any gains from the transform. In \cite{leo11} it is therefore proposed to concentrate on orthogonal transforms which have cost of the order $O(n\log(n))$ or less.

Examples of such {\em fast orthogonal transforms} include discrete sine and cosine transform, Walsh and (inverse) Haar transform as well as the orthogonal matrix corresponding to the PCA, see Scheicher \cite{schei} and Leobacher \cite{leo11}.  
\\

A relatively recent approach to enhance the efficiency of Monte Carlo simulation has been proposed by  Giles \cite{giles08} and Heinrich \cite{heinrich01}. They propose a multilevel procedure by combining  Monte Carlo  based on different time discretizations. The improvement in computational efficiency by using quasi-Monte Carlo instead of Monte Carlo together with the multilevel method is shown in Giles \& Waterhouse \cite{giles09} where the authors used a rank-1 lattice rule with a random shift. Furthermore, they give a short discussion on the three classical sampling methods mentioned above. We contribute to the topic by finding an orthogonal transform adapted to the multilevel method, thus making it even more efficient.
\\

The remainder of the paper is organized as follows. Section \ref{sec:householder} reviews basic properties of Householder reflections and in Section \ref{sec:regression} we describe an algorithm for finding a fast orthogonal transform using Householder reflections. The main part of our article, Section \ref{sec:mlqmc}, recalls some of the basics of multilevel (quasi-)Monte Carlo and  discusses how the ideas of Section \ref{sec:regression} can be  carried over to multilevel quasi-Monte Carlo integration.

Section \ref{sec:numeric} gives a numerical example where the method described earlier is applied to an example from finance. We will see that the method improves the efficency of multilevel quasi-Monte Carlo integration.

\section{Householder Reflections}\label{sec:householder}

We recall the definition and basic properties of Householder reflections from Golub \& Van Loan \cite{golub96}.
\begin{definition}
A matrix of the form
\begin{align*}
U =I-2 \frac{v v^\top}{v^\top\!v} \,,
\end{align*}
where $v\in\R^n$, is called a {\em Householder reflection}. The vector $v$ is called the defining {\em Householder vector}.
\end{definition}

In the following proposition, $e_1$ denotes the first canonical basis vector in $\R^n$, $e_1=(1,0,\ldots,0)$.

\begin{proposition}
Householder reflections have the following properties:
\begin{enumerate}
\item Let $U$ be a Householder reflection with Householder vector $v$. If $x\in\R^n$ is a vector then $Ux$ is the reflection of $x$ in the hyperplane $\mathrm{span}\{v\}^\perp$. In particular, $U$ is orthogonal and symmetric, i.e.\ $U^{-1}=U$.
\item Given any vector $a\in\R^n$ we can find $v\in\R^n$ such that for the corresponding Householder reflection $U$ we have $Ua=\|a\|e_1$. The computation of the Householder vector uses $3n$ floating point operations.
\item The computation of $Ux$ uses at most $4n$ floating point operations.
\end{enumerate}
\end{proposition}

\begin{proof}
See Golub \& Van Loan \cite[Chapter 5.1]{golub96}.
\end{proof}

\section{Regression Algorithm}\label{sec:regression}

In this section we give a short description of a rather general method for constructing fast and efficient orthogonal transforms. Parts of the material have already been presented in \cite{il12a}, but we include them to make the paper self-contained.

Let $f:\R^n\longrightarrow\R$ be a measurable function with $\EE(f(X)^2)<\infty$ for a standard normal vector $X$. Wang \& Sloan \cite{sw10} consider functions of the form 
\begin{align}\label{eq:ws}
f(X)=g(w_1^\top X,\ldots,w_m^\top X)
\end{align}
with $w_1,\ldots, w_m\in\R^n$ and $g:\R^m\longrightarrow\R$. The authors show that for such functions there exists an orthogonal transform that reduces the dimension of $f$ to at most $m$. Therefore the integration problem is not as high-dimensional as it seems and we have a convergence rate of the QMC algorithm applied to the transformed problem corresponding to $m$ rather than $n$. We give a slightly modified version of their arguments to reduce the dimension of $f$, because we suggest using Householder reflections to generate the orthogonal transform which guarantees that the transform can be applied using at most $O(n\log(n))$ operations if $m\leq\log(n)$. 

Assume that $w_1$ is not the zero vector and let $U_1:\R^n\rightarrow\R^n$ be a Householder reflection which maps $e_1$ to $w_1/\|w_1\|$. Then $w_1^\top U_1 X=\|w_1\| e_1^\top X=\|w_1\| X_1$ and therefore 
\begin{align*}
f(U_1 X)=g(\|w_1\| X_1,(U_1w_2)^\top  X,\ldots,(U_1 w_m)^\top X).
\end{align*}
Next we write $(U_1 w_k)^\top X=(U_1 w_k)^\top_1 X_1+(U_1w_k)^\top_{2\ldots n} X_{2\ldots n}$. That is, 
\begin{align*}
f(U_1X)=g_{1}(X_1, w_{2,1}^\top X_{2\ldots n},\ldots,w_{m,1}^\top X_{2\ldots n})
\end{align*}
where $w_{2,1},\ldots,w_{m,1}\in\R^{n-1}$. Assuming that $w_{2,1}\ne 0$, let $\tilde U_2:\R^{n-1}\rightarrow\R^{n-1}$ be the Householder reflection that maps $e_1$ to $w_{2,1}/\|w_{2,1}\|$ and let
\begin{align*}
U_2=\left(\begin{array}{ccc}1&0\\0&\tilde U_2\end{array}\right)\,.
\end{align*}
Then $U_2$ is a Householder reflection from $\R^n$ to $\R^{n}$ and
\begin{align*}
f(U_1U_2X)=g_{2}(X_1,X_2,w_{3,2}^{\;\top} X_{3\ldots n},\ldots,w_{m,2}^{\;\top} X_{3\ldots n})\,
\end{align*}
with $w_{3,2},\dots,w_{m,2}\in\R^{n-2}$. Proceeding that way one arrives at 
\begin{align*}
f(U_1\cdots U_{\hat m} X)=g_{\hat m}(X_1,X_2,\ldots,X_{\hat m})\,
\end{align*}
for some $\hat m\le m$ (We may have $\hat m<m$ if some transformed $w_k$ are zero).\\

In the spirit of \cite{sw10} we propose a procedure for more general integration problems. Let us assume that the function $f$ is of the form
\begin{align*}
f(x)=\tilde g(h_1(x),\ldots,h_m(x))
\end{align*}
where $m<n$, $\tilde g:\R^m\longrightarrow\R$ and $h_k:\R^n\longrightarrow\R,~k=1,\ldots,m$. We want to approximate every $h_k$ by a linear function, i.e.
\begin{align*}
h_k(x)\approx a_k^\top x+b_k
\end{align*}
with $a_k\in\R^{n}$ and $b_k\in\R$. The approximation is done by a ``linear regression'' approach and therefore, for every $k=1,\ldots,m$, we minimize 
\begin{align*}
\EE\left(\left(h_k(X)-a_k^\top X-b_k\right)^2\right)\rightarrow\min.
\end{align*}
First order conditions give for $k=1,\ldots,m$
\begin{align}
a_{k,j}&=\EE\left(X_j h_k(X)\right),\quad j=1,\dots,n\,;\label{eq:firstord1}\\
b_k&=\EE\left(h_k(X)\right)\,.\label{eq:firstord2}
\end{align}
Therefore, (\ref{eq:firstord1})-(\ref{eq:firstord2}) minimizes the variance of the difference between each $h_k(X)$ and its linear approximation $a_k^\top X+b_k$. So
\begin{align*}
\V\bigl(h_k(X)\bigr)&=\EE\Bigl(\bigl(h_k(X)-b_k\bigr)^2\Bigr)\\
&=\EE\Bigl(\bigl(a_k^\top X\bigr)^2+\bigl(h_k(X)-b_k-a_k^\top X\bigr)^2\Bigr)\\
&=\bigl\|a_k\bigr\|^2+\V\Bigl(h_k(X)-a_k^\top X\Bigr)\;.
\end{align*}
That is, $\|a_k\|^2/\V(h_k(X))$ measures the fraction of variance captured by the linear approximation. 

Now we approximate the function $f$ by substituting the $h_k$ with the linear functions obtained by the linear regression, i.e.
\begin{align*}
f(x)\approx\tilde g(a_1^\top x+b_1,\ldots,a_m^\top x+b_m)=g(a_1^\top x,\ldots,a_m^\top x).
\end{align*}
Therefore $f$ is approximated by a function of the form (\ref{eq:ws}) and we can proceed in the same way as at the beginning of this section to determine a fast orthogonal transform by using Householder reflections. 

Note that the method is only practical if the expectations $\EE(X_j h_k(X))$ in (\ref{eq:firstord1}) can be computed explicitly or at least efficiently. After the statement of the algorithm we will give an example where explicit calculation is possible.

\begin{algorithm}\label{alg:general}
Let $X_1,\ldots,X_n$ be independent standard normal variables. Let $f$ be a function $f:\R^n\longrightarrow \R$, which is of the form $f=g\circ h$ where $h:\R^n\longrightarrow \R^m$ and $g:\R^m\longrightarrow \R$.
\begin{enumerate}
\item Start with $k,\ell=1$ and $U=I$;
\item\label{it:goto} $a_{k,j}:=\EE(X_j h_k(UX))$ for $j=k,\ldots,n$;
\item $a_{k,j}:=0$ for $j=1,\ldots,k-1$;
\item if $\|a_k\|=0$ go to \ref{it:jump};
\item else let $U_\ell$ be a Householder reflection that maps 
$e_\ell$ to $a_k/\|a_k\|$;
\item $U=U U_\ell$; $\ell=\ell+1$;
\item \label{it:jump}$k=k+1$;
\item while $k\leq m$, go back to \ref{it:goto};
\item\label{it:final} Compute $\EE(f(UX))$ using QMC.
\end{enumerate}
\end{algorithm}

\begin{example}\label{ex:asian} We give an example from finance for which Algorithm \ref{alg:general} can be applied efficiently. Motivated by a discrete arithmetic Asian option let us consider
\begin{align*}
f(X)=\max\left(\sum_{k=1}^{n}w_{k}\exp\left(\sum_{i=1}^{n}(c_{k,i}X_i+d_{k,i})\right)-K,0\right).
\end{align*}
with $w_k, c_{k,i}, d_{k,i}\in\R$. Now we can write $f(X)=g(h_1(X))$ with $g(y)=\max(y-K,0)$ and $h_1(X)=\sum_{k=1}^{n}w_k\exp(\sum_{i=1}^{n}(c_{k,i}X_i+d_{k,i}))$. In that case we can compute $\EE(X_j h_1(X))$ explicitly. It is easily verified that, with $\phi$ denoting the standard normal density, $\phi(x)=\exp(-x^2/2)/\sqrt{2\pi}$,
\begin{align*}
\int_{\R}\exp(cx+d)\phi(x)dx=\exp(c^2/2+d)
\end{align*}
and
\begin{align*}
\int_{\R}x\exp(cx+d)\phi(x)dx=c\exp(c^2/2+d)
\end{align*}
for any $c,d\in\R$. Therefore, we obtain
\begin{align*}
a_{1,j}&=\EE(X_{j}h_{1}(X))\\
&=\int_{\R}\dots\int_{\R}x_j h_1(x)\phi(x_1)\dots\phi(x_n)dx_1\dots dx_n\\
&=\sum_{k=1}^{n}w_k c_{k,j}\exp\left(\sum_{i=1}^{n}\frac{c_{k,i}^2}{2}+d_{k,i}\right).
\end{align*}
In \cite{il12a} it was calculated that for practical parameters $\|a\|^2$ is typically larger than $0.99\cdot \V(h_1(X))$. 
\end{example}

\section{Multilevel Quasi-Monte Carlo}\label{sec:mlqmc}

We start with an abstract formulation of the multilevel (quasi-)Monte Carlo method: suppose we want to approximate $\EE(Y)$ for some random variable $Y$ which has finite expectation. Suppose further that we have a sequence of sufficiently regular functions $f^\ell:\R^{m^\ell}\rightarrow \R$ such that 
\begin{align}\label{eq:mc_approx}
\lim_{\ell\rightarrow\infty} \EE(f^{\ell}(X^\ell))=\EE(Y)\,,
\end{align}
where for each $\ell\ge 0$, $X^\ell$ denotes an $m^\ell$-dimensional standard normal vector. (\ref{eq:mc_approx}) states that there exists a sequence of algorithms which approximate $\EE(Y)$ with increasing accuracy. For example, if $f^\ell(X^\ell)$ has finite variance, we can approximate $\EE(Y)$ by $\frac{1}{N}\sum_{k=0}^{N-1} f^\ell(X^\ell_k)$  using sufficiently large $\ell$ and $N$, where $(X^\ell_k)_{k\ge 0}$ is a sequence of independent standard normal vectors.

Usually, evaluation of $f^\ell(X^\ell)$ becomes more costly with increasing $\ell$ and $N$. Multilevel methods sometimes help us to save significant proportions of computing time by computing more samples for the coarser approximations, which need less computing time but have higher variance.

We will need the following definition in the statement of the multilevel Monte Carlo method: for any $m\in\N$ and any $\ell\in\N_0$ we call the $m^{\ell-1}\times m^\ell$ matrix $C_{m,\ell}=\bigl((C_{m,\ell})_{i,j}\bigr)_{i,j}$ with 
\begin{align*}
(C_{m,\ell})_{i,j}:=\left\{\begin{array}{cl}\frac{1}{\sqrt{m}}&\quad\mbox{ if }~(i-1)m+1\leq j\leq i\,m\\0& \quad\mbox{ else }\end{array}\right.
\end{align*}
the {\em coarsening matrix} from level $\ell$ to level $\ell-1$. For example, in the case of $m=2$ the coarsening matrix is given by
\begin{align*}
C_{2,\ell}:=\left(\begin{array}{cccccccc}\frac{1}{\sqrt{2}}&\frac{1}{\sqrt{2}}&0&0&0&\dots&0&0\\
0&0&\frac{1}{\sqrt{2}}&\frac{1}{\sqrt{2}}&0&\dots&0&0\\
\vdots&\vdots&\vdots&\vdots&\vdots& &\vdots&\vdots\\
0&0&0&0&0&\dots&\frac{1}{\sqrt{2}}&\frac{1}{\sqrt{2}}
\end{array}\right)\,.
\end{align*}

The following lemma is simple to verify and therefore we leave the proof to the reader.

\begin{lemma} 
Let $m\in\N$. If $X^\ell$ is an $m^\ell$--dimensional standard normal vector, then $C_{m,\ell}X^\ell$ is an $m^{\ell-1}$--dimensional standard normal vector.\flushright\qed
\end{lemma}

Obviously we have 
\begin{align}
\EE(Y)\approx\EE\left(f^L(X^L)\right)&=\EE\left(f^0(X^0)\right)+\sum_{\ell=1}^L \EE\left(f^\ell(X^\ell)\right)-\EE\left(f^{\ell-1}(X^{\ell-1})\right)\nonumber\\
&=\EE\left(f^0(X^0)\right)+\sum_{\ell=1}^L \EE\left(f^\ell(X^\ell)\right)-\EE\left(f^{\ell-1}(C_{m,\ell}X^{\ell})\right)\nonumber\\
&=\EE\left(f^0(X^0)\right)+\sum_{\ell=1}^L \EE\left(f^\ell(X^\ell)-f^{\ell-1}(C_{m,\ell}X^{\ell})\right)\label{eq:mlmcrep}
\end{align}

Equation (\ref{eq:mlmcrep}) becomes useful if, as is often the case in practice, the expectation $\EE\left(f^\ell(X^\ell)-f^{\ell-1}(C_{m,\ell}X^{\ell})\right)$ can be approximated to the required level of accuracy using less function evaluations for bigger $\ell$ while the costs per function evaluation increases. One typical situation where this occurs is when a stochastic differential equation is solved numerically using time discretization with $m^\ell$ time steps and $f^\ell$ is some function on the set of solution paths. See \cite{giles08} for how to exploit this representation.

In finance, $f^\ell$ is typically of the form $f^\ell(X)=\psi(h^\ell(X))$ for some functions $h^\ell:\R^{m^\ell}\longrightarrow\R$  and $\psi:\R\longrightarrow\R$. In that context, $h^\ell$ is some function taking as its argument a discrete (geometric) Brownian path, like the maximum or the average, and $\psi$ is the payoff that depends on the outcome of $h^\ell$.
\begin{align*}
\EE\left(f^L(X^L)\right)&=\EE\left(f^0(X^0)\right)+\sum_{\ell=1}^L \EE\left(\psi\left(h^\ell(X^\ell)\right)-\psi\left(h^{\ell-1}(C_{m,\ell}X^{\ell})\right)\right)\\
&=\EE\left(\psi\left(h^0(X^0)\right)\right)+\sum_{\ell=1}^L \EE\left(g^\ell\bigl(h^\ell_1(X^\ell),h^\ell_2(X^{\ell})\bigr)\right)\,,
\end{align*}
where $h_1^\ell=h^\ell$, $h_2^\ell=h^{\ell-1}\circ C_{m,l}$ and $g^\ell(y_1,y_2)=\psi(y_1)-\psi(y_2)$.

Now the integrands are precisely of the form covered by Algorithm \ref{alg:general}. It is therefore sensible to apply the corresponding orthogonal transform $U^\ell:\R^\ell\longrightarrow\R^\ell$ at each level such that we get \begin{align*}
\EE\left(f^{L}(X^{L})\right)=\EE\left(f^0(U^0X^0)\right)+\sum_{\ell=1}^L \EE\left(g^\ell\left(h^\ell_1(U^\ell X^\ell),h^\ell_2(U^\ell X^{\ell})\right)\right)\,.
\end{align*}
Of course, we are free to try any other set of orthogonal transforms, like PCA. The advantage of using the regression algorithm is that here at each level the orthogonal transform is determined by taking both the fine and the coarse discretization into account. 

In the next section we shall try our method on a concrete example from finance, the Asian option.

\section{Asian Option}\label{sec:numeric}

We will consider an Asian call option in the Black-Scholes model, i.e.\ under the risk-neutral measure the stock price process $S=(S_{t})_{t\geq0}$ is given by the stochastic differential equation (SDE) 
\begin{align*}
dS_{t}=rS_{t}dt+\sigma S_{t}dB_{t}
\end{align*}
where $r$ is the interest rate, $\sigma$ is the volatility and $(B_{t})_{t\geq0}$ is a standard Brownian motion. Given the stock price $S_0$ at time $0$, the solution of the SDE is given by 
\begin{align*}
S_{t}=S_{0}\exp\left(\left(r-\sigma^2/2\right)t+\sigma B_{t}\right).
\end{align*}
The payoff of the Asian call option with fixed strike price $K$, maturity $T$ and underlying $S$ is 
\begin{align*}
\max\left(\frac{1}{T}\int_{0}^{T}S_{t}dt-K,0\right).
\end{align*}
That is, if the average stock price over the time interval $[0,T]$ is above level $K$, the option pays its holder at time $T$ the difference between that value and $K$, otherwise it pays nothing.

Martingale pricing theory tells us that the price of the option is given by the discounted expectation of the payoff function under the risk-neutral measure, see Bj\"ork \cite[Chapter 10]{bjoerk}, i.e.
\begin{align*}
C=\exp(-rT)\EE\left(\max\left(\frac{1}{T}\int_0^TS_tdt-K,0\right)\right)\,.
\end{align*}
To approximate $C$ in the time-continuous model, a common way is to use (multilevel) quasi-Monte Carlo integration to compute the expectation. To that end we first approximate the integral by a sum: For any equidistant time discretization with $n\in\N$ points, 
\begin{align*}
\frac{1}{T}\int_0^T S_tdt\approx \frac{1}{n}\sum_{k=1}^nS_k(X)
\end{align*}
where $X=(X_1,\ldots,X_n)$ is a standard normal vector and
\begin{align*}
S_{k}(X)=S_{0}\exp\left(\left(r-\frac{\sigma^2}{2}\right)k\frac{T}{n}+\sigma\sqrt{\frac{T}{n}}\sum_{i=1}^{k}X_i\right),\quad k=1,\ldots,n\,.
\end{align*}
Therefore the payoff function of the Asian option is approximately 
\begin{align}\label{eq:discretepayoff}
f(X)=\max\left(\frac{1}{n}\sum_{k=1}^{n}S_{k}(X)-K,0\right).
\end{align}
If the time discretization consists of $\ell=m^\ell$ points with $\ell\in\N_0$, $m\in\N$, we denote the payoff function by $f^\ell$ and it is therefore given by (\ref{eq:discretepayoff}) with $n=m^{\ell}$.

Thus we can approximate the price $C$ using multilevel QMC integration with finest level $L$ by
\begin{align*}
C&\approx\exp(-rT)\left(\EE\left(f^0(X^0)\right)+\sum_{k=1}^{L}\EE\left(f^{\ell}(X^\ell)-f^{\ell-1}(C_{m,\ell}X^\ell)\right)\right)\\
&=\exp(-rT)\left(\EE\left(f^0(X^0)\right)+\sum_{k=1}^{L}\EE\left(g^{\ell}(h_1^\ell(X^\ell),h_2^\ell(X^\ell))\right)\right)
\end{align*}
where $g^\ell(y_1,y_2)=\max(y_1,0)-\max(y_2,0)$,
\begin{align*}
h_1^\ell(X^\ell)=\frac{1}{m^\ell}\sum_{k=1}^{m^\ell}S_0\exp\left(\left(r-\frac{\sigma^2}{2}\right)k\frac{T}{m^\ell}+\sigma\sqrt{\frac{T}{m^\ell}}\sum_{i=1}^{k}X_i^\ell\right)
\end{align*}
and
\begin{align*}
h_2^\ell(X^\ell)=\frac{1}{m^{\ell-1}}\sum_{k=1}^{m^{\ell-1}}S_0\exp\left(\left(r-\frac{\sigma^2}{2}\right)k\frac{T}{m^{\ell-1}}+\sigma\sqrt{\frac{T}{m^{\ell-1}}}\sum_{i=1}^{k}(C_{m,\ell}X^\ell)_i\right)\,.
\end{align*}
In applying Algorithm \ref{alg:general}, we have to compute the vectors $a_{1}^\ell, a_{2}^\ell$ for each level. This can be done as in Example \ref{ex:asian}. For $\ell=1,\ldots,L$ and $j=1\dots,m^\ell$ we get 
\begin{align*}
a_{1,j}^{\ell}&=\EE\left(X_{j}h_{1}^{(l)}\right)\\
&=\sum_{k=1}^{m^\ell}\frac{S_{0}}{m^\ell}\exp\left(\Bigl(r-\frac{\sigma^2}{2}\Bigr)k\frac{T}{m^\ell}\right)\EE\left(X_{j}\exp\biggl(\sigma\sqrt{\frac{T}{m^\ell}}\sum_{i=1}^{k}X_{i}\biggr)\right)\\
&=\sum_{k=j}^{m^\ell}\frac{S_{0}}{m^\ell}\exp\left(\Bigl(r-\frac{\sigma^2}{2}\Bigr)k\frac{T}{m^\ell}\right)\sigma\sqrt{\frac{T}{m^\ell}}\exp\left(\frac{\sigma^2}{2}\frac{T}{m^\ell}k\right)\\
&=\sum_{k=j}^{m^\ell}\frac{S_{0}\sigma}{m^\ell}\sqrt{\frac{T}{m^\ell}}\exp\left(rk\,\frac{T}{m^\ell}\right)
\end{align*}
and
\begin{align*}
a_{2,j}^{\ell}&=\EE\left(X_{j}h_{2}^{(\ell)}\right)\\
&=\sum_{k=1}^{m^{\ell-1}}\frac{S_{0}}{m^{\ell-1}}\exp\left(\Bigl(r-\frac{\sigma^2}{2}\Bigr)k\frac{T}{m^{\ell-1}}\right)\EE\left(X_{j}\exp\biggl(\sigma\sqrt{\frac{T}{m^{\ell}}}\sum_{i=1}^{k}\sum_{p=1}^{m}X_{(i-1)m+p}\biggr)\right)\\
&=\sum_{k=\left\lfloor\frac{j-1}{m}\right\rfloor+1}^{m^{\ell-1}}\frac{S_{0}}{m^{\ell-1}}\exp\left(\Bigl(r-\frac{\sigma^2}{2}\Bigr)k\frac{T}{m^{\ell-1}}\right)\sigma\sqrt{\frac{T}{m^\ell}}\exp\left(\frac{\sigma^2}{2}\frac{T}{m^{\ell-1}}k\right)\\
&=\sum_{k=\left\lfloor\frac{j-1}{m}\right\rfloor+1}^{m^{\ell-1}}\frac{S_{0}\sigma}{m^{\ell-1}}\sqrt{\frac{T}{m^\ell}}\exp\left(rk\,\frac{T}{m^{\ell-1}}\right)\,.
\end{align*}

Now we compare the multilevel QMC method combined with the regression algorithm with multilevel Monte Carlo and multilevel quasi-Monte Carlo (forward and PCA sampling) numerically. For that we choose the parameters as $r=0.04$, $\sigma=0.3$, $S_0=100$, $K=100$ and $T=1$. At the finest level we start with $2^{10}$ discretization points and at each coarser level we divide in half the number of points, i.e. $L=10$ and $m=2$. Furthermore, the number of sample points are doubled at each level starting with $N_L$ sample points at the finest level $L$. For the QMC approaches we take a Sobol sequence with a random shift. In Table \ref{tbl:asian1} we compare for different values $N_L$ both the average and the standard deviation of the price of the Asian call option based on $1000$ independent runs. Moreover, the average computing time for one run is given in brackets. As we can see, the regression algorithm yields the lowest standard deviation, but the computing time of the regression algorithm is slightly worse than the forward method. However, the regression algorithm is better than the PCA construction measured in both standard deviation and computing time.

\begin{table}[ht]
\setlength{\tabcolsep}{4pt}
\setlength{\extrarowheight}{1pt}
\begin{center}
{\small
\begin{tabular}{l|cc|c|ccccc}
&	\multicolumn{2}{c|}{multilevel}    				&\multicolumn{6}{c}{multilevel QMC}\\
&	\multicolumn{2}{c|}{Monte Carlo}		&\multicolumn{2}{c}{forward}		&\multicolumn{2}{c}{PCA} &\multicolumn{2}{c}{regression}\\\hline
$N_L$&average&stddev&average&stddev&average&stddev&average&stddev \\\hline\hline
2	 &7.717& $0.41\!\times\!10^{0}$	  &7.735& $0.19\!\times\!10^{-1}$	  &7.736& $0.16\!\times\!10^{-1}$	  &7.739& $0.10\!\times\!10^{-1}$\\
&\multicolumn{2}{c|}{(0.0057\,s)}		&\multicolumn{2}{c}{(0.0057\,s)}	&\multicolumn{2}{c}{(0.0088\,s)}	&\multicolumn{2}{c}{(0.0069\,s)}\\\hline
4	 &7.738& $0.19\!\times\!10^{0}$   &7.734& $0.71\!\times\!10^{-2}$   &7.736& $0.44\!\times\!10^{-2}$   &7.738& $0.29\!\times\!10^{-2}$\\
& \multicolumn{2}{c|}{(0.0074\,s)}  &\multicolumn{2}{c}{(0.0074\,s)}  &\multicolumn{2}{c}{(0.0118\,s)}  &\multicolumn{2}{c}{(0.0091\,s)}\\\hline
8	 &7.748& $0.54\!\times\!10^{-1}$	&7.737& $0.30\!\times\!10^{-2}$	  &7.737& $0.14\!\times\!10^{-2}$	  &7.736& $0.10\!\times\!10^{-2}$\\
& \multicolumn{2}{c|}{(0.0101\,s)}	& \multicolumn{2}{c}{(0.0100\,s)}	& \multicolumn{2}{c}{(0.0165\,s)}	& \multicolumn{2}{c}{(0.0124\,s)}\\\hline
16 &7.746& $0.40\!\times\!10^{-1}$	&7.736& $0.11\!\times\! 10^{-2}$  &7.737& $0.69\!\times\! 10^{-3}$  &7.736& $0.30\!\times\! 10^{-3}$\\
& \multicolumn{2}{c|}{(0.0157\,s)}	& \multicolumn{2}{c}{(0.0157\,s)} & \multicolumn{2}{c}{(0.0279\,s)}	& \multicolumn{2}{c}{(0.0194\,s)}\\\hline
32 &7.728& $0.31\!\times\!10^{-1}$	&7.736& $0.49\!\times\! 10^{-3}$	&7.737& $0.21\!\times\! 10^{-3}$	&7.736& $0.10\!\times\! 10^{-3}$\\
& \multicolumn{2}{c|}{(0.0266\,s)}	& \multicolumn{2}{c}{(0.0265\,s)} & \multicolumn{2}{c}{(0.0585\,s)}	& \multicolumn{2}{c}{(0.0326\,s)}\\\hline
64 &7.739& $0.81\!\times\!10^{-2}$	&7.736& $0.20\!\times\! 10^{-3}$  &7.737& $0.69\!\times\! 10^{-4}$	&7.737& $0.32\!\times\! 10^{-4}$\\
& \multicolumn{2}{c|}{(0.0486\,s)}	& \multicolumn{2}{c}{(0.0484\,s)}	& \multicolumn{2}{c}{(0.1202\,s)} & \multicolumn{2}{c}{(0.0583\,s)}\\
\end{tabular}}
\caption{Multilevel (Q)MC using $2^{10}$ time steps $(L=10)$. The average and the standard deviation of the option price are based on $1000$ runs. The average computing time is given in brackets.}\label{tbl:asian1}
\end{center}
\end{table}

In Table \ref{tbl:asian2} we compare the regression algorithm both for multilevel QMC and for QMC with $2^{10}$ time steps $(L=10)$. We can observe that the standard deviation as well as the computing time of the multilevel QMC setting is significantly better compared with crude QMC.

\begin{table}[ht]
\setlength{\tabcolsep}{4pt}
\setlength{\extrarowheight}{1pt}
\begin{center}
\begin{tabular}{lc|ccc}
										&  						 & \quad average\quad~ & \quad stddev\quad~	& \quad time (s)\quad~   \\\hline\hline
MLQMC - Regression	& ($N_L=2^6$)		& 7.7366	& $0.32\!\times\!10^{-4}$		& 0.0323 \\
QMC - Regression 		& ($N=2^{12}$)	& 7.7362	& $1.01\!\times\!10^{-4}$		& 0.1511 \\ 
\end{tabular}
\caption{QMC and multilevel QMC, both combined with the regression algorithm, with $2^{10}$ time steps ($L=10$) based on $1000$ runs.}\label{tbl:asian2}
\end{center}
\end{table}

\end{document}